\documentclass{amsart}

\makeatletter
\@namedef{subjclassname@2020}{%
  \textup{2020} Mathematics Subject Classification}
\makeatother

\usepackage{tikz}
\tikzset{every picture/.style={line width=0.11mm}}
\newcommand{\oPerpSymbol}{\begin{tikzpicture}[scale=0.134]
  \draw (0,-0.5)--(0,1); \draw (-0.866,-0.5)--(0.866,-0.5);
  \draw (0,0) circle [radius=1];
\end{tikzpicture}}

\newcommand{\oPerp}{\mathbin{\raisebox{-1pt}{\oPerpSymbol}}}

\usepackage{graphicx}

\usepackage{hyperref}

\usepackage{amssymb}

\usepackage{amsthm}

\usepackage{amsmath}
\usepackage{textcomp}

\usepackage{calrsfs}

\DeclareMathAlphabet{\pazocal}{OMS}{zplm}{m}{n}

\def\esc#1{\langle #1\rangle}
\def\escd#1{\langle\!\langle #1\rangle\!\rangle}
\def\F{\mathcal F}
\newcommand{\K}{{\mathbb{K}}}
\newcommand{\FC}{{\mathbb{F}}}
\def\G{\mathcal G}
\def\H{\mathcal H}
\def\FF{\mathbf{F}}
\def\R{\mathbb{R}}
\def\Q{\mathbb{Q}}
\def\C{\mathbb{C}}
\def\rad{\mathop{\hbox{\bf rad}}}
\def\rads{\mathop{\hbox{\footnotesize\bf rad}}}
\def\diag{\mathop{\hbox{\rm diag}}}

\def\l{\lambda}

\def\sp{\mathop{\hbox{\rm span}}}
\def\so{simultaneously orthogonalizable}

\def\T{\mathcal T}
\def\D{\mathcal D}

\def\span{\mathop{\hbox{\rm span}}}
\def\L{\mathcal{L}}
\newtheorem{lemma}{Lemma}
\newtheorem{proposition}{Proposition}
\newtheorem{corollary}{Corollary}
\newtheorem{definition}{Definition}
\newtheorem{theorem}{Theorem}
\newtheorem{example}{Example}
\newtheorem{remark}{Remark}

\title[Simultaneous orthogonalization of inner products]{Simultaneous orthogonalization of inner products over arbitrary fields}
\author[Y. Cabrera]{Yolanda Cabrera Casado}

\author[C. Gil]{Crist\'obal Gil Canto}

\author[D. Mart\'{\i}n]{Dolores Mart\'in Barquero}

\author[C. Mart\'{\i}n]{C\'andido Mart\'in Gonz\'alez}

\address{Departamento de Matem\'atica Aplicada, E.T.S. Ingenier\'\i a Inform\'atica, Universidad de M\'alaga, Campus de Teatinos s/n. 29071 M\'alaga.   Spain. }
\email{yolandacc@uma.es}

\address{Departamento de Matem\'atica Aplicada, E.T.S. Ingenier\'\i a Inform\'atica, Universidad de M\'alaga, Campus de Teatinos s/n. 29071 M\'alaga.   Spain.}
\email{cgilc@uma.es}

\address{Departamento de Matem\'atica Aplicada, Escuela de Ingenier\'\i as Industriales, Universidad de M\'alaga, Campus de Teatinos s/n. 29071 M\'alaga.   Spain.}
\email{dmartin@uma.es}

\address{ Departamento de \'Algebra Geometr\'{\i}a y Topolog\'{\i}a, Fa\-cultad de Ciencias, Universidad de M\'alaga, Campus de Teatinos s/n. 29071 M\'alaga.   Spain.}
\email{candido\_m@uma.es}

\thanks{The  authors are supported by the Spanish Ministerio de Ciencia e Innovaci\'on through the project  PID2019-104236GB-I00 and by the Junta de Andaluc\'{\i}a  through the projects  FQM-336 and UMA18-FEDERJA-119,  all of them with FEDER funds.}

\usepackage{amsmath}

\begin{document}


\maketitle

\begin{abstract}
We give necessary and sufficient conditions for a family of inner products in a finite-dimensional vector space $V$ over an arbitrary field $\K$ to have an orthogonal basis relative to all the inner products. Some applications to evolution algebras are also considered.
\end{abstract}



\section{Introduction}

The problem of determining when a finite collection of symmetric matrices is simultaneously diagonalizable via congruence is historically well-known. The first known result is that two real symmetric matrices $A$ and $B$ can be simultaneously diagonalizable via real congruence if one of them, $A$ or $B$, is definite. Another classical attempt to solve the problem when we have a family of two inner products is as follows: if given two symmetric $n \times n$ matrices $A$ and $B$ over the reals which satisfy that there exist $\alpha, \beta \in \R$ such that $\alpha A + \beta B$ is positive definite, then $A$ and $B$ are simultaneously diagonalizable via congruence. Next, the previous statement was proved to be an ``if and only if" for the case $n \geq 3$ by Finsler (1937) \cite{Finsler} and Calaby (1964) \cite{Calabi}. Afterwards Greub \cite{Greub} finds a necessary condition in order to have two real symmetric matrices $A$ and $B$ simultaneously diagonalizable via congruence for $n \geq 3$, specifically: $(xAx^t)^2+(xBx^t)^2\neq0$ if $x\neq0$. Besides Wonenburger (1966) and Becker (1978), among other authors, give necessary and sufficient conditions for a pair of symmetric matrices to be simultaneously diagonalizable via congruence in a more general setting. The study that most resembles ours is Wonenburger's \cite{Wonenburger} who studies first the case of two nondegenerate symmetric bilinear forms (for a field of characteristic $\ne 2$) and in a subsequent result she attacks the problem when the ground field is real closed  and the forms do not vanish simultaneously. In particular our Corollary \ref{rabia} generalises \cite[Theorem 1]{Wonenburger}. However, Becker \cite{Becker} approaches his work about simultaneously diagonalization via congruence of two hermitian matrices with $\C$ as a base field. Going a bit further, Uhlig \cite{Uhligart} works the simultaneously block diagonalization via congruence of two real symmetric matrices. Part of the literature concerning the simultaneously diagonalization via congruence of only pairs of quadratic forms are included in the survey \cite{Uhlig} and for more general information about this topic see Problem $12$ in \cite{HU}.

In the fourth's author talk \lq\lq Two constructions related to evolution algebras\rq\rq\  given in 2018's event \lq\lq Research School on Evolution Algebras and non associative algebraic structures\rq\rq\ \cite{school} at University of M\'alaga (Spain), the idea of the simultaneous orthogonalization of inners products in the context of evolution algebras was exposed. The question posed there was: how to classify up to simultaneous congruence couples of simultaneous orthogonalizable inner products.

Recently the work \cite{BMV} deals also with the problem of simultaneously orthogonalization of inner products on finite-dimensional vector spaces over $\K=\R$ or $\C$. One of the motivations in \cite{BMV} is that of detecting when a given algebra
is an evolution algebra. The issue is equivalent to the simultaneous orthogonalization of a collection of inner products (the projections of the product in
the lines generated by each element of a given basis of the algebra). 
Modulo some minimal mistake in \cite[Theorem 2]{BMV} and \cite[Corollary 1]{BMV} (where the \lq\lq if and only if\rq\rq\ is actually a one-direction implication), the paper provides answers to the question working over $\C$. However, we will see in our work that it is possible
to handle satisfactorily the task considering general ground fields. 
Though we solve the problem for arbitrary fields, the case of characteristic two, has to be considered on its own in some of our results.
 
The paper is organized as follows. In Section \ref{tarde} we introduce the preliminary definitions and notations. If we consider an algebra of dimension $n$,
we introduce in subsection \ref{edrat} the evolution test ideal.
This is an ideal $J$ in certain polynomial algebra on $n^2+1$ indeterminates 
whose zero set $V(J)$ tell us if $A$ is or not an evolution algebra. 
In section \ref{S:1} we consider the nondegenerate case: we have a family of
inner products $\F$ such that at least one of the inner products is nondegenerate.
Theorem \ref{nomam} is the main result in this section, the example below would be helpful to understand how it works. Our result includes the characteristic $2$ case and we have inserted Example \ref{3olpmeje} to illustrate how to deal with this eventuality. In section \ref{Miller} we deal with
the general case (that in which no inner product in the family is nondegenerate). We introduce a notion of radical of a family of inner products inspired by the general
procedure of modding out those elements which obstruct the symmetry
of a phenomenon. After examining the historical development of the topic of the simultaneously diagonalization via congruence, we realized that the idea appears, of course, in other cited works
dealing with the problem (see for example \cite{Becker} and \cite{Wonenburger}). In this way we reduce the problem to families whose radical is zero (though possibly degenerate). We define also a notion of equivalence of families of inner products (roughly speaking two such families are equivalent when the orthogonalizability of the space relative to one of them is equivalent to the orthogonalizability relative to the second family, and moreover  each orthogonal decomposition relative to one of families is also an orthogonal decomposition relative to the other).
When the ground field is infinite we prove that any $\F$ with zero radical we can modified to an equivalent nondegenerate family $\F'$ by adding only one inner product. Finally Theorem \ref{erdamus} and Proposition \ref{tsal} can be considered
as structure theorem of families $\F$ of inner products with zero radical. In order to better understand how to apply the thesis given in the last results we refer to Example \ref{otxes}, where we work with an algebra of dimension $6$ over the rationals.

\section{Preliminaries and basic results}\label{tarde}
Let $V$ be a vector space over a field $\K$, a family $\F=\{\esc{\cdot,\cdot}_i\}_{i\in I}$ of inner products (symmetric bilinear forms) on $V$ is said to be \emph{simultaneously orthogonalizable} if there exists a basis $\{v_j\}_{j\in\Lambda}$ of $V$ such that for any $i\in I$ we have   
$\esc{v_j,v_k}_i=0$ whenever $j\ne k$.  We say  that the family $\F$ is \emph{nondegenerate} if there exists $i \in I$ such that $\esc{\cdot,\cdot}_i$ is nondegenerate. Otherwise, we say that the family $\F$ is \emph{degenerate}. Let $V$ be a vector space over a field $\K$ with a inner product $\esc{\cdot,\!\cdot}\colon V\times V \to \K$ and $T \colon V \to V$ a linear map.  We will say that $T^\sharp\colon V\to V$ is the \emph{adjoint} of $T\colon V\to V$ when $\esc{T(x),y}=\esc{x,T^\sharp(y)}$ for any $x,y\in V$. A linear map $T\colon V\to V$ is said to be \emph{self-adjoint} when its adjoint exists and $T^\sharp=T$. 
The existence of the adjoint is guaranteed in the case of a nondegenerate inner product on a finite-dimensional vector space. 

If $V$ is a finite-dimensional $\K$-vector space and $\esc{\cdot,\!\cdot}$ an inner product in $V$, it is well-known that if the characteristic of $\K$ ($\hbox{char}(\K)$) is other than $2$, there is an orthogonal basis of $V$ relative to $\esc{\cdot,\!\cdot}$. In the characteristic two case there are inner products with no orthogonal basis, for instance, the inner product of matrix $\tiny\begin{pmatrix}0 & 1\cr 1 & 0\end{pmatrix}$ in the vector space $\FF_2^2$ over the field of two elements $\FF_2$. But even over fields of characteristic other than two, there are sets of inner products which are not simultaneously orthogonalizable. Such is the case of the inner products of the real vector space $\R^2$ which in the canonical basis have the matrices: $$\left(
\begin{array}{cc}
 1 & 0 \\
 0 & 1 \\
\end{array} 
\right),\quad \left(
\begin{array}{cc}
 0 & 1 \\
 1 & 0 \\
\end{array}
\right), \quad \left(
\begin{array}{rr}
 1 & 0 \\
 0 & -1 \\
\end{array}
\right).$$

\begin{remark}\label{atupoi} \rm
{Concerning the above example, if $\F$ is a family of inner products on a finite-dimensional vector space $V$ of dimension $n$ and $\F$ is \so, then the maximum number
of linearly independent elements of $\F$ is $n$. Thus if we have a family of more than $n$ linearly independent inner products on an $n$-dimensional vector space, we can conclude that $\F$ is not \so. This is the reason why the above three inner products are not \so.
}
\end{remark}

If $\F=\{\esc{\cdot \ , \cdot }_i\}_{i\in I}$ is a family of inner products on a space $V$ over a field $\K$, it may happen that $\F$ is not simultaneously orthogonalizable but that under scalar extension to a field $\FC\supset \K$, the family $\F$ considered as a family of inner products in the scalar extension $V_\FC$ is simultaneously orthogonalizable as we show in the following example. 
\begin{example}\label{ejem1} \rm \label{gm}
Consider the inner products of the real vector space $\R^2$ whose Gram matrices relative to the canonical basis are

$$
\left(
\begin{array}{rr}-1 & 1\cr 1 & 1\end{array}\right), \left(\begin{array}{rr}1 & 1\cr 1 & -1\end{array}\right).
$$
It is easy to check that they are not simultaneously orthogonalizable,  however the inner products of the complex vector space $\C^2$ with the same matrices are simultaneously orthogonalizable relative to the basis $B=\{(i,1),(-i,1)\}$. 
\end{example}

One of the applications of \so\  theory is to evolution algebras. For this reason, we briefly remember some basic notions related to these algebras. A $\K$-algebra $A$ is an \emph{evolution algebra} if there exists a basis $B=\{e_i\}_{i\in I}$ called \emph{natural basis} such that $e_ie_j=0$ for any $i\neq j$. Fixed a natural basis $B$ of $A$, the matrix $C=(c_{ij})$ with $c_{ij} \in \K$ such that $e_i^2=\sum_jc_{ji}e_j$ will be called the \emph{structure matrix} of $A$ relative to $B$.

The problem of simultaneously orthogonalization of families of inner products in a vector space is directly related to that of detecting whether or not a given algebra is an evolution algebra. If $A$ is a commutative algebra over a field $\K$ the product in $A$ can be written in the form 
\begin{equation}\label{mrof}
x y=\sum_{i \in I}\esc{x,y}_i e_i
\end{equation}
where $\{e_i\}_{i\in I}$ is any fixed basis of $A$ and the inner products $\esc{\cdot,\cdot}_i\colon A\times A\to \K$ provide the coordinates of $xy$ relative to the basis. So $A$ is an evolution algebra if and only if the set of inner products $\esc{\cdot,\cdot}_i$ is simultaneously orthogonalizable. 
The above result appears in \cite[Theorem 1]{BMV} under the additional conditions that the dimension of the algebra is finite and the ground field is $\R$ or $\C$.
Note that the ``if and only if" condition of \cite[Theorem 2]{BMV} is not true (see Example \ref{gm}). The implication that is trivially true is that 
$\F=\{\esc{\cdot \ , \cdot }_i\}_{i\in I}$ being simultaneously orthogonalizable in the vector space $V$ over $\K$, implies that the same family considered in the scalar extension $V_\FC$ (with $\FC \supset \K$) is also simultaneously orthogonalizable. Note also that \cite[Corollary 2]{BMV} is wrong.
%


\subsection{Evolution test ideal}\label{edrat}

We will show that given a commutative finite-dim\-en\-sional algebra $A$ over a field $\K$ there is an ideal in a certain polynomial algebra which detects if $A$ is an evolution algebra. This will allow to apply tools of algebraic geometry in this setting:  consider the affine space $\K^m$ and $H$
an ideal of the polynomial $\K$-algebra $\K[x_1,\ldots,x_m]$. We
define the set of zeros of $H$ as the set of common zeros of the polynomials in $H$, that is,  $V(H):=\{x\in\K^m\colon q(x)=0,\ \forall q\in H\}$ (here we use the duality polynomial-function).

Assume $\dim(A)=n$ and consider the polynomial algebra in the $n^2+1$ indeterminates of the set $\{z\}\sqcup \{x_{ij}\}_{i,j=1}^n$, so we have the polynomial algebra $R:=\K[\{z\}\sqcup \{x_{ij}\}_{i,j=1}^n]$. Fix now a basis $\{e_i\}_{i=1}^n$ of $A$ and write the product of $A$ in the form given in equation \eqref{mrof} which gives the family of inner products $\{\esc{\cdot,\cdot}_i\}_{i=1}^n$. Define $M_k$ as the Gram matrix of the inner product $\esc{\cdot,\cdot}_k$ in any basis (the same basis for all the inner products). 

\begin{definition} \rm
In the above conditions, we consider the ideal $J$ of $R$ generated by the polynomials $p_0$ and $p_{ijk}$ defined as: 
$$
\begin{array}{rl}
p_0(x_{11},\ldots,x_{nn},z) & :=  1-z\det[(x_{ij})_{i,j=1}^n]
\end{array}
$$
and
$$
\begin{array}{rl}
      p_{ijk}(x_{i1},\ldots,x_{in},x_{j1},\ldots,x_{jn}) &:=(x_{i1},\ldots,x_{in})M_k(x_{j1},\ldots,x_{jn})^t
\end{array}
$$
where $i,j,k=1,\ldots,n$, $i\neq j$.
The ideal $J$ will be called the {\em evolution test ideal} of $A$ though its form depends of the chosen basis.

\end{definition}

The name of this ideal $J$ is motivated by the following proposition.

\begin{proposition} Let $A$ be a $n$-dimensional commutative algebra over a field $\K$ and $J$ the evolution test ideal of $A$ (fixed a basis). Then $A$ is an evolution algebra if and only if $V(J)\ne\emptyset$. 
\end{proposition}
\begin{proof}
Write the product of $A$ in the form $xy=\sum_{k=1}^n \esc{x,y}_ke_k$ where $B=\{e_k\}_{k=1}^n$ is a basis of $A$. Let $M_k$ be the matrix of each $\esc{\cdot,\cdot}_k$ relative to $B$ for $k=1,\ldots n$. 
If $A$ is an evolution algebra, fix a natural basis $\{u_q\}_{q=1}^n$ of $A$. For $q=1,\ldots, n$, assume that the coordinates of  $u_q$ relative to $B$ are $(u_{q1},\ldots,u_{qn})$. Since $\esc{u_p,u_q}_k=0$ if $p\ne q$, we have  
$p_{ijk}(u_{i1},\ldots,u_{in},u_{j1},\ldots,u_{jn})=0$ for $i\ne j$
and $p_0(u_{11},\ldots,u_{nn},\Delta)=0$ where $\Delta=\det[(u_{ij})_{i,j=1}^n]$. Consequently 
$$(u_{11},\ldots, u_{nn},\Delta)\in V(J).$$
Conversely if $(u_{11},\ldots, u_{nn},\Delta)\in V(J)$, then 
$p_0(u_{11},\ldots, u_{nn},\Delta)=0$ which implies that 
the vectors $u_q:=(u_{q1},\ldots,u_{qn})$ (for $q=1,\ldots,n$) are a basis
of the vector space. Furthermore, we also have $p_{ijk}(u_{i1},\ldots,u_{in},u_{j1},\ldots,u_{jn})=0$ for $i,j,k\in\{1,\ldots, n\}$ and $i\ne j$. This tell us that if $i\ne j$, the vectors $u_i$ and $u_j$ are orthogonal relative to $\esc{\cdot,\cdot}_k$ for any $k$.
Whence $\{u_q\}_{q=1}^n$ is a natural basis of $A$.
\end{proof}

If we have $1\in J$, then $V(J)=\emptyset$ hence $A$ is not an evolution algebra.
In the case of an algebraically closed field $\K$ the condition $1\in J$ is equivalent to $V(J)=\emptyset$ by the Hilbert's Nullstellensatz. 

To see how this works in low dimension, consider the following example:

\begin{example} \rm
Take an evolution algebra over a field $\K$ with product given by the inner products of matrices as in Example \ref{ejem1}.
The evolution test ideal is
the ideal $J$ of $\K[x_{11},x_{12},x_{21},x_{22},z]$ generated by the polynomials:
$$\begin{matrix}
   -x_{11}
   x_{21}+x_{12}
   x_{21}+x_{11}
   x_{22}+x_{12}
   x_{22}\cr x_{11}
   x_{21}+x_{12}
   x_{21}+x_{11}
   x_{22}-x_{12} x_{22}\cr -z
   x_{12} x_{21}+z x_{11}
   x_{22}-1.
   \end{matrix}$$
In case $\text{char}(\K)=2$ the ideal is just 
the generated by 
$$x_{11}
   x_{21}+x_{12}
   x_{21}+x_{11}
   x_{22}+x_{12}
   x_{22},\quad z
   x_{12} x_{21}+z x_{11}
   x_{22}+1.$$
 For $\K={\mathbf F}_2$ the zeros of the polynomial above must  have $z=1$, $x_{12}x_{21}+x_{11}x_{22}=1$ hence  $x_{11}
   x_{21}+x_{12}
   x_{22}=1$ 
  and we have (among others) the solution $x_{11}=x_{12}=1$, $x_{21}=0$, $x_{22}=1$. Since any field of characteristic two contains ${\mathbf F}_2$ we find that $A$ is an evolution algebra when the ground field is of characteristic two. 
  If $\K$ has characteristic other than $2$, a Gro\"ebner basis of $J$ is the set $$\left\{x_{21}^2+x_{22}^2, 2 z
   x_{12} x_{22}^2-x_{21}, 2
   z x_{12} x_{21}+1, 2 z
   x_{22}
   x_{12}^2+x_{11}\right\}.$$
Observe that if $(x_{11},x_{12},x_{21},x_{22},z)$ is a zero of the polynomials, then $x_{21}\ne 0$ which implies $x_{22}\ne 0$. In this case, $\frac{x_{21}}{x_{22}}$ has to be a square root of $-1$. Consequently if $\sqrt{-1}\notin\K$ the algebra is not an evolution algebra. If $\sqrt{-1}\in\K$, then we have a solution $x_{11}=x_{21}=1$, $x_{12}=-x_{22}={\bf i}:=\sqrt{-1}$ and $z=\frac{\bf i}{2}$. So the zero set of the evolution test ideal says that $A$ is  an evolution algebra only when $\sqrt{-1}\in\K$.
\end{example}
The evolution test ideal of a three-dimensional algebra could involve a set of  $1$  quartic polynomial of $10$ variables and $9$ quadratic polynomials of $6$ variables.  In this case the computations are more involved and it seems reasonable to use another tools
to elucidate if a given algebra is an evolution algebra or not.

\section{The nondegenerate case}\label{S:1}

Let $\K$ be a field and $V$ an $n$-dimensional vector space over $\K$. Let $I$ be a nonempty set and for $i\in I$, let $\esc{\cdot \ , \cdot }_i$ be an inner product on $V$, that is, symmetric bilinear form $\esc{\cdot\  , \cdot }_i \ \colon V\times V\to \K$. The question is: under what conditions is the family $\{\esc{\cdot,\!\cdot}_i\}_{i\in I}$ simultaneously orthogonalizable? 
In other words, under what conditions is there a basis $\{v_j\}_{j=1}^n$ of $V$ such that for any $i\in I$ we have   
$\esc{v_j,v_k}_i=0$ whenever $j\ne k$?
Of course each inner product must be orthogonalizable hence this is a necessary condition. 
As a first approach we will assume that one of the inner products is nondegenerate and we will denote this as $\esc{\cdot,\!\cdot}_0$. 
Under this assumption there is a canonical isomorphism $V\cong V^*$, where $V^*$ is the dual space of $V$. The isomorphism is defined by  
$x\mapsto\esc{x,\_}_0$. 
In the next lemma we find necessary and sufficient conditions to ensure that a family of inner products is \so. In this case the ground field $\K$ is arbitrary. 

\begin{lemma}\label{ogacem}
Assume that $\F=\{\esc{\cdot,\!\cdot}_i\}_{i\in I\cup\{0\}}$ is a family of inner products in the finite-dimensional vector space $V$ over $\K$, and that $\esc{\cdot,\!\cdot}_0$ is nondegenerate. Then:
\begin{enumerate}
\item \label{ogacem1} For each $i\in I$
there is a linear map $T_i\colon V\to V$ such that $\esc{x,y}_i=\esc{T_i(x),y}_0$ \ for any $x,y\in V$. Furthermore, each $T_i$ is a self-adjoint operator of $(V,\esc{\cdot , \cdot}_0)$.
\item \label{ogacem2}$\F$ is \so\  if and only if there exists an orthogonal basis $B$ of $(V,\esc{\cdot , \cdot}_0)$ such that each $T_i$ is diagonalizable relative to $B$.
\end{enumerate}
\end{lemma}

\begin{proof}
To prove item \ref{ogacem1}, fix $x\in V$ and $i \in I$ and consider the element 
$\esc{x,\_}_i\in V^*$. From the isomorphism $V\cong V^*$ above, we conclude that  there is a unique $a_x^i\in V$ such that $\esc{x,\_}_i=\esc{a_x^i,\_}_0$. So for any $y\in V$ we have $\esc{x,y}_i=\esc{a_x^i,y}_0$ for every $i \in I$.
Now, for any $i\in I$ define $T_i\colon V\to V$ by $T_i(x)=a_x^i$. For proving the linearity of $T_i$ take into account that 
$$\esc{a_{x+y}^i-a_x^i-a_y^i,\_}_0=\esc{x+y,\_}_i-\esc{x,\_}_i-\esc{y,\_}_i=0,$$
$$\esc{a_{\lambda x}^i-\lambda a_x^i,\_}_0=\esc{\lambda x,\_}_i-\lambda \esc{x,\_}_i=0,$$
so nondegeneracy of $\esc{\cdot,\!\cdot}_0$ gives the linearity of each $T_i$. Now, let us prove that $\esc{T_i(x),y}_0=\esc{x,T_i(y)}_0$ for any $x,y\in V$ and $i \in I$:
$$\esc{T_i(x),y}_0=\esc{x,y}_i=\esc{y,x}_i=
\esc{T_i(y),x}_0=\esc{x,T_i(y)}_0.$$
Now we prove item \ref{ogacem2},  assume that $\F$ is simultaneously orthogonalizable. Let $B=\{v_j\}$ be a basis of $V$ with $\esc{v_j,v_k}_i=0$ for $j\ne k$ and any $i\in I\cup\{0\}$. 
For $i \in I$ we write $T_i(v_j)=\sum_k a_{ij}^k v_k$ we have 
$$\esc{T_i(v_j)-a_{ij}^j v_j, v_k}_0=\esc{T_i(v_j), v_k}_0-a_{ij}^j\esc{ v_j, v_k}_0 = \esc{v_j, v_k}_i=0 \hbox{ if $k\ne j$ }.$$
$$\esc{T_i(v_j)-a_{ij}^j v_j, v_j}_0= \sum_{q\ne j}a_{ij}^q\esc{v_q,v_j}_0=0.$$ And then
$T_i(v_j)\in \K v_j$ for arbitrary 
$i \in I$ and $j$. Thus each self-adjoint operator $T_i$ is diagonalizable in the basis $B$. So we have proved that
there is an orthogonal basis of $V$ relative to $\esc{\cdot,\!\cdot}_0$ such that each $T_i$ diagonalizes relative to that basis.
Reciprocally, assume that for any $i\in I$ we have that each $T_i$ is diagonalizable relative to a certain orthogonal basis $B=\{v_j\}$ respect to $\esc{\cdot,\!\cdot}_0$. 
Thus $T_i(v_j)\in \K v_j$ and 
we can write $T_i(v_j)=a_{ij}v_j$ for some $a_{ij}\in \K$.
So $\F$ is simultaneously orthogonalizable in $B$ since for any $i,j,k$ with $j\ne k$ we have: 
$$\esc{v_j,v_k}_i=\esc{T_i(v_j),v_k}_0=a_{ij}\esc{v_j,v_k}_0=0.$$
\end{proof}

\begin{remark} \rm
In the conditions of Lemma \ref{ogacem} note that $T_iT_j=T_jT_i$ for any $i,j\in I$.
\end{remark}

A well known result that we will apply in the sequel is that for a finite-dimensional vector space $V$, any commutative family of diagonalizable linear maps  $\{T_i\}_{i\in I}$ is simultaneously diagonalizable in the sense that there is a basis $B$ of $V$ such that each $T_i$ in the family is diagonalizable relative to $B$.

We use the symbol $\oPerp$ to denote orthogonal direct sum. 
If $\F$ is a family of inner products in a vector space $V$ and
we have subspaces $S,T\subset V$ such that $V=S\oplus T$ and $S$, $T$ are orthogonal relative to all the inner products in $\F$, we will use the notation $V=S\oPerp_{\tiny\F} T$. The symbol $\oPerp_{\F}$ will be abbreviated to $\oPerp$ if no confusion is
possible.

\begin{lemma}\label{eneiton}
Let $V$ be a finite-dimensional vector space over a field $\K$ of characteristic other than $2$, endowed with a nondegenerate inner product $\esc{\cdot,\!\cdot}\colon V\times V\to \K$. Assume that $\mathcal{T}$ is an commutative family of self-adjoint  diagonalizable linear maps of $V\to V$. Then there is an orthogonal basis $B$ of $(V,\esc{\cdot,\!\cdot})$ such that all the elements of $\mathcal{T}$ are diagonal relative to $B$.
\end{lemma}
\begin{proof}
We know that there is a basis $C=\{v_1,\ldots,,v_n\}$ of $V$ such that each $T\in\mathcal{T}$ diagonalizes with respect to $C$. If $C$ is an orthogonal basis we are done. Otherwise we consider the following set of families 
$$
\begin{array}{cc}
   \D = \left\{\{V_i\}_{i \in I} \text{ a family of vector subspaces of $V$} \  \vert \ \ V=\oplus_{i\in I} V_i \  \text{and} \right.\nonumber\\ \qquad \left. {} \ T\vert_{V_i}=\lambda_i(T) 1_{V_i} \ \  \forall \ T\in\mathcal{T}\right\}.   
\end{array}
$$ 
Observe that $\D \neq \emptyset$ because $\{\K v_i\}_{i=1}^n \in \D$. Let  $\{V_i\}_{i \in I}$ be a family in $\D$ such that the cardinal of $I$ is minimum. Take $i,j\in I$ different. We will prove that there exists $T\in  \T$ such that $\lambda_i(T) \neq \lambda_j (T)$. Indeed, if for any $T\in  \T$ we have $\lambda_i(T) =\lambda_j (T)$, then we redefine the decomposition of $V$ in direct sum of vector subspaces in the following way: let $J:=(I\setminus\{i,j\})\sqcup\{q\}$. Then the cardinal of $J$ is lower than the cardinal of $I$. Next define
$$\begin{cases}W_q:=V_i \oplus V_j,\cr W_k:=V_k\ (k \neq i,j)\end{cases}$$ 
Note that $T\vert_{W_j}=\mu_j(T)1_{W_j}$ for some scalars $\mu_j(T)\in\K$. 
Then $\{W_j\}_{j\in J}\in\D$ and the cardinal of $J$ is lower than the cardinal of $I$, a contradiction. Hence there is some $T\in\T$ such that $\l_i(T)\ne\l_j(T)$. Let us check that $V_i\bot V_j$: take $0\ne x\in V_i$ and $0\ne y\in V_j$, then
$$\l_i(T)\esc{x,y}=\esc{T(x),y}=\esc{x,T(y)}=\l_j(T)\esc{x,y}$$
whence $\esc{x,y}=0$. Thus we have $V=\oPerp_{i\in} V_i$ and since the characteristic of the ground field is not $2$, each $V_i$ has an orthogonal basis. So we are done.
\end{proof}

Now we can summarize Lemma \ref{ogacem} and Lemma \ref{eneiton} as follows:

\begin{theorem}\label{nomam}
Assume that $\F=\{\esc{\cdot,\!\cdot}_i\}_{i\in I\cup\{0\}}$ is a family of inner products in the finite-dimensional vector space $V$ over a field $\K$ and that $\esc{\cdot,\!\cdot}_0$ is nondegenerate. Then for each $i\in I$
there is a linear map $T_i\colon V\to V$ such that $\esc{x,y}_i=\esc{T_i(x),y}_0$ \ for any $x,y\in V$. Furthermore, each $T_i$ is a self-adjoint operator of $(V,\esc{\cdot , \cdot}_0)$. We also have
\begin{enumerate}
\item  The family $\F$ is simultaneously orthogonalizable if and only if each $T_i$ is diagonalizable relative to an orthogonal basis of $V$ relative to $\esc{\cdot,\!\cdot}_0$.
\item If $\hbox{\rm char}(\K)\ne 2$ the family $\F$  is simultaneously orthogonalizable if and only if $\{T_i\}_{i\in I}$ is a commutative family of  diagonalizable endomorphisms of $V$.
\end{enumerate}
\end{theorem}

\begin{remark} \rm
Observe that the basis $B$ relative to which $\F$ is simultaneously orthogonalizable, coincides with the basis diagonalizing each $T_i$ with $i \in I$. 
\end{remark}

In the context of the above Lemma \ref{eneiton} if we fix a basis 
$B=\{v_j\}$ of $V$,  we will denote by $M_B(T_i)$ the matrices of $T_i$ relative to $B$. 
This means that $M_B(T_i)=(a_{ij}^k)_{j,k}$ where $T_i(v_j)=\sum_k a_{ij}^k v_k$ for any $i$ and $j$.
On the other hand, the matrices $M_{i,B}:=(\esc{v_j,v_k}_i)_{j,k}$ of the inner products $\esc{\cdot,\!\cdot}_i$ in $B$ are related by the
equations  
$\esc{v_j,v_t}_i=\esc{T_i(v_j),v_t}_0=
\sum_k a_{ij}^k\esc{v_k,v_t}_0$, that is, 
$M_{i,B}=M_B(T_i)M_{0,B}$. Equivalently $M_B(T_i)=M_{i,B}M_{0,B}^{-1}$. Summarizing, we have the following.

\begin{corollary}\label{rabia}
Fix a basis $B$ of a vector space $V$ of finite dimension over a field $\K$ with $\hbox{\rm char}(\K)\ne 2$ and assume that $\F=\{\esc{\cdot,\!\cdot}_i\}_{i\in I\cup\{0\}}$ is a family of inner products on $V$ whose matrices in $B$ are $M_{i,B}$. Further assume that $M_{0,B}$ is nonsingular. Then $\F$ is simultaneously orthogonalizable if and only if 
the collection of matrices $\{M_{i,B}M_{0,B}^{-1}\}_{i\in I}$ is commutative and 
each one of them is diagonalizable.
\end{corollary}

Next, we illustrate Theorem \ref{nomam} and Corollary \ref{rabia} with two examples over a field $\K$, one of them with $\hbox{\rm char}(\K)\ne 2$ and another one with $\hbox{\rm char}(\K)= 2$.
\begin{example}\rm
Consider the inner products given by the matrices
$$M_0=\left(
\begin{array}{ccc}
 4 & 2 & 3 \\
 2 & 2 & 1 \\
 3 & 1 & 3 \\
\end{array}
\right),\quad M_1=\left(
\begin{array}{ccc}
 0 & 0 & 1 \\
 0 & 0 & 1 \\
 1 & 1 & 1 \\
\end{array}
\right),\quad M_2=\left(
\begin{array}{ccc}
 -3 & -2 & -2 \\
 -2 & -2 & -1 \\
 -2 & -1 & -2 \\
\end{array}
\right)$$
relative to a certain basis $B$ of $\K^3$. Assume that the characteristic of $\K$ is other than $2$ so that $M_0$ is nonsingular (we shall investigate the singular case later on). Are these inner products simultaneously orthogonalizable?  We have: 
$$M_1M_0^{-1}=\left(
\begin{array}{ccc}
 -2 & 1 & 2 \\
 -2 & 1 & 2 \\
 -1 & 1 & 1 \\
\end{array}
\right) ,\quad 
M_2M_0^{-1}=\left(
\begin{array}{rrr}
 -\frac{1}{2} & -\frac{1}{2} & 0 \\
 0 & -1 & 0 \\
 \frac{1}{2} & -\frac{1}{2} & -1 \\
\end{array}
\right)$$
which can be seen to commute. Moreover $M_1M_0^{-1}$ is diagonalizable since its minimal polynomial is 
$-x(x+1)(x-1)$. Also $M_2M_0^{-1}$ is diagonalizable its minimal polynomial being $-(x+1)(x+\frac{1}{2})$. Thus, there is basis which is orthogonal relative to the three inner products.
If we want to find a basis orthogonalizing all the inner product it suffices to diagonalize simultanously the matrices $M_1M_0^{-1}$ and 
$M_2M_0^{-1}$. For the first one, the eigenspace of eigenvalue $0$ is 
generated  by $v_1=(1,-1,0)$, the one of eigenvalue $1$ is generated by $v_2=(-1,1,1)$ and that of eigenvalue $-1$ by $v_3=(1,0,-1)$. 
But $v_1M_2M_0^{-1}=-\frac{1}{2}v_1$ while $v_2M_2M_0^{-1}=-v_2$ and $v_3M_2M_0^{-1}= -v_3$.
Then the matrices of the inner products in the basis $\{v_1, v_2, v_3\}$ are
$$\left(
\begin{array}{ccc}
 2 & 0 & 0 \\
 0 & 1 & 0 \\
 0 & 0 & 1 \\
\end{array}
\right),\quad 
\left(
\begin{array}{rrr}
 0 & 0 & 0 \\
 0 & 1 & 0 \\
 0 & 0 & -1 \\
\end{array}
\right)\hbox{ and }
\left(
\begin{array}{rrr}
 -1 & 0 & 0 \\
 0 & -1 & 0 \\
 0 & 0 & -1 \\
\end{array}
\right).$$
\end{example}

Since our methods include also the characteristic two case we can handle an example like the following.

\begin{example}\label{3olpmeje}  \rm

Let $\K$ be a field of characteristic two and $V=\K^3$. Let 
$\F$ be the family of inner products whose Gram matrices relative to the canonical
basis are:
$$M_0=\left(
\begin{array}{ccc}
 1 & 0 & 0 \\
 0 & 0 & 1 \\
 0 & 1 & 0 \\
\end{array}
\right),\ M_1=\left(
\begin{array}{ccc}
 0 & 0 & 1 \\
 0 & 0 & 1 \\
 1 & 1 & 1 \\
\end{array}
\right),\ M_2=\left(
\begin{array}{ccc}
 0 & 1 & 0 \\
 1 & 1 & 1 \\
 0 & 1 & 0 \\
\end{array}
\right).$$
Consider the inner product on $V$ given by $\esc{x,y}_0:=xM_0y^t$. Define next
the linear maps $T_1,T_2\colon V\to V$ given by $T_i(x)=xM_iM_0^{-1}$ for $i=1,2$.
We have
$$M_1M_0^{-1}=\left(
\begin{array}{ccc}
 0 & 1 & 0 \\
 0 & 1 & 0 \\
 1 & 1 & 1 \\
\end{array}
\right),\ M_2M_0^{-1}=\left(
\begin{array}{ccc}
 0 & 0 & 1 \\
 1 & 1 & 1 \\
 0 & 0 & 1 \\
\end{array}
\right)$$
and both matrices are diagonalizable, being a basis of simultaneous eigenvectors for both matrices $v_1=(1,0,1)$, $v_2=(1,1,0)$, $v_3=(1,1,1)$. Since these vectors
are pairwise orthogonal relative to $\esc{\cdot,\cdot}_0$, Theorem \ref{nomam}(1) implies that $\F$ is \so. An orthogonal basis for $\F$ is $B=\{v_1,v_2,v_3\}$ and the Gram matrices of the $M_i$'s relative to $B$ are:
$$\left(
\begin{array}{ccc}
 1 & 0 & 0 \\
 0 & 1 & 0 \\
 0 & 0 & 1 \\
\end{array}
\right), \left(
\begin{array}{ccc}
 1 & 0 & 0 \\
 0 & 0 & 0 \\
 0 & 0 & 1 \\
\end{array}
\right), \left(
\begin{array}{ccc}
 0 & 0 & 0 \\
 0 & 1 & 0 \\
 0 & 0 & 1 \\
\end{array}
\right).$$
\end{example}

\section{The degenerate case}\label{Miller}

Assume as before that $\K$ is a field and $V$ a vector space over $\K$. Recall that for an inner product $\esc{\cdot,\!\cdot}\colon V\times V\to \K$, we denote by $V^{\perp}$ the subspace $V^{\perp}:=\{x\in V\colon\esc{x,V}=0\}$ that we will call the {\it radical} of the inner product. We will also use the notation $\rad(V,\esc{\cdot,\cdot})$ for $V^\bot$ if we want to be more accurate. If there is no possible confusion we will use $\rad(\esc{\cdot,\cdot})$. If there is an orthogonal basis $\{v_i\}$ of $V$ relative to an inner product $\esc{\cdot,\!\cdot}\colon V\times V\to \K$, then the radical of the inner product is the linear span of all the $v_i$'s such that $\esc{v_i,v_i}=0$.

 Let $I$ be a nonempty set and for $i\in I$, let $\F=\{\esc{\cdot,\!\cdot}_i\}_{i\in I}$ be a family of inner products on $V$ and consider $\rad(V,\esc{\cdot,\cdot}_i)$. We define the {\it radical of the family} $\F$ by $\rad(\F):=\cap_{i\in I}\rad(V,\esc{\cdot,\cdot}_i)$.
There is a subspace $W$ of $V$ such that $V=\rad(\F)\oplus W$ and $\esc{\rad(\F),W}_i=0$ for any $i\in I$. Indeed, 
any subspace $W$ complementing $\rad(\F)$ satisfies $\esc{\rad(\F),W}_i=0$ for any $i$. So we can write $V=\rad(\F)\oPerp_\F W$. Then we have the following result.


\begin{proposition}\label{tomato}
Let $V$ be a vector space of arbitrary dimension over a field $\K$. Let  $\F=\{\esc{\cdot,\!\cdot}_i\}_{i\in I}$ be a family of inner products on $V$.
There is a subspace $W$ of $V$ (in fact any complement of $\rad(\F)$) such that $V=\rad(\F)\oPerp_\F W$. Moreover, the family of inner products $\F\vert_W=\{\esc{\cdot,\!\cdot}_i\vert_W\}_{i\in I}$ on $W$ satisfies 
$\rad(\F\vert_W)=0$
and the collection 
$\F$ is simultaneously orthogonalizable if and only if $\F\vert_W$ is simultaneously orthogonalizable.
\end{proposition}

\begin{proof}
Observe that by the definition of $\rad(\F)$ we have  
$$\rad(\F\vert_W)=0.$$
Indeed, if $x\in W$ satisfying $\esc{x,W}_i=0$ for any $i\in I$, then $\esc{x,V}_i=\esc{x,\rad(\F)}_i+\esc{x,W}_i=0$
implying $x\in\rad(\F)$. So $x=0$. Clearly, if the family $\F\vert_{W}$ is simultaneously orthogonalizable, then $\F$ is simultaneously orthogonalizable. Conversely take a basis  $\{e_j\}_{j\in J}$ of $V$ such that for any $j,k \in J$ with $j\ne k$ and $i \in I$ verifying $\esc{e_j,e_k}_i=0$. For any $j\in J$ write $e_j=r_j+w_j$ with $r_j\in\rad(\F)$ and $w_j\in W$. We have 
  $$0=\esc{e_j,e_k}_i=\esc{w_j,w_k}_i$$ which proves that the collection of vectors $\{w_j\}_{j\in J}$ is orthogonal relative to any inner product $\esc{\cdot,\!\cdot}_i$.
  Now define the set $J_1:=\{j\in J\colon w_j\ne 0\}$, we see that $\{w_j\}_{j\in J_1}$ is a basis of $W$. First, we show that it is a system of generators of $W$: take $w\in W$ then $w=\sum_j \l_je_j=\sum_j \l_j r_j+\sum_j \l_j w_j$, ($\l_j\in \K$). Thus $W\ni w-\sum_j\l_j w_j=\sum_j\l_j r_j\in\rad(\F)$ hence $w=\sum_j\l_j w_j$. In order to prove the linear independence, consider scalars $\l_j$ and assume $\sum_{j\in J_1}\l_j w_j=0$. Then for any $i\in I$ and $k\in J_1$ we have $0=\sum_{j\in J_1}\l_j\esc{w_j,w_k}_i=\l_k\esc{w_k,w_k}_i$.
  So if  $\l_k\ne 0$, then $\esc{w_k,w_k}_i=0$ for any $i$. Therefore  $w_k\in \rad({\F\vert_{W}})=0 $ whence $w_k=0$, a contradiction (because $k\in J_1$). Thus $\{w_j\}_{j\in J_1}$ is an orthogonal basis of $W$ relative to any $\esc{\cdot,\!\cdot}_i$.
\end{proof}

So we can reduce the problem of orthogonalizing a collection of inner products $\F=\{\esc{\cdot,\!\cdot}\}_{i\in I}$ to the case
in which $\rad(\F)=0$. 

\begin{example}\rm
Consider the vector space $\K^4$ and the family $\F$ on inner products given by the matrices below: 
$$N_1=
\left(
\begin{array}{rrrr}
3 & -2 & -2 & 0 \\
-2 & 2 & 1 & -1 \\
-2 & 1 & 2 & 1 \\
0 & -1 & 1 & 2 \\
\end{array}
\right),\ N_2=
\left(
\begin{array}{rrrr}
0 & 0 & -1 & -1 \\
0 & 0 & 1 & 1 \\
-1 & 1 & 1 & 0 \\
-1 & 1 & 0 & -1 \\
\end{array}
\right)$$
$$N_3=\left(
\begin{array}{rrrr}
3 & -1 & -3 & -2 \\
-1 & 1 & 1 & 0 \\
-3 & 1 & 3 & 2 \\
-2 & 0 & 2 & 2 \\
\end{array}
\right),\
N_4=\left(
\begin{array}{rrrr}
2 & -1 & -1 & 0 \\
-1 & 1 & 0 & -1 \\
-1 & 0 & 1 & 1 \\
0 & -1 & 1 & 2 \\
\end{array}
\right).
$$
Now, we analyze if the family $\F$ is \so. For a generic $v\in \K^4$, solving the equations $vN_i=0$ ($i=1,2,3,4$) we find that $\rad(\F)=\K(0,1,-1,1)$ so $\K^4=\rad(\F)\oplus W$ where $W$ can be taken to be the linear span of $$e_1=(1,0,1,1),\ e_2=(1,1,0,1),\ e_3=(1,1,1,0).$$ The restriction of the inner products to $W$ is given by the linearly independent matrices
$$M_0=\left(
\begin{array}{ccc} 5 & 2 & 0\cr 2 & 1 & 0\cr 0 & 0 & 1\end{array}\right),\ 
M_1=\left(
\begin{array}{rrr}-4 & -2 & 0\cr -2 & -1 & 0\cr 0 & 0 & 1\end{array}\right),\
M_2=\left(\begin{array}{ccc}
 2 & 0 & 0\cr 0 & 0 & 0\cr 0 & 0 & 1\end{array}\right),$$
relative to the basis $\{e_1,e_2,e_3\}$ of $W$. By Remark \ref{atupoi}, observe that the maximum number of linearly independent inner products on $W$ has to be three 
if we want $W$ to have an orthogonal basis relative to them.
Since $\vert M_0\vert=1$ we can apply the procedure explained in Corollary \ref{rabia}. Then 
$$\small M_0^{-1}=\left(
\begin{array}{rrrr}
1 & -2 & 0 \\
-2 & 5 & 0 \\
0 & 0 & 1\\
\end{array}
\right),\ M_1M_0^{-1}= \left(
\begin{array}{rrrr}
0 & -2 & 0 \\
0 & -1 & 0 \\
0 & 0 & 1\\
\end{array}
\right),\ M_2M_0^{-1}= \left(
\begin{array}{rrrr}
2 & -4 & 0 \\
0 & 0 & 0 \\
0 & 0 & 1\\
\end{array}
\right),$$
and it can be checked that $M_1M_0^{-1}$ commutes with $M_2M_0^{-1}$ and that both matrices are diagonalizable since their characteristic polynomials are $x(x+1)(x-1)$ and $x(x-1)(x-2)$. We conclude that $\F$ is simultaneously orthogonalizable. If we want to find a basis which is orthogonal relative to $\F$ we first find a basis of $W$ which orthogonalizes the inner products of matrices $M_0$, $M_1$ and $M_2$. For this, it suffices to simultaneously diagonalize the matrices  $A_i=M_iM_0^{-1}$, with $i=1,2$. A basis of common eigenvectors for both matrices is 
$\{e_1-2 e_2,e_2,e_3\}$. Moreover defining
$$v_0:=(0,-1,1,-1), v_1:=e_1-2e_2=(1,0,1,1)-2(1,1,0,1)=(-1,-2,1,-1),$$
$$v_2:=e_2=(1,1,0,1),\ v_3:=e_3=(1,1,1,0),$$ we get a basis $\{v_i\}_{i=0}^3$ of $\K^4$ such that the matrices of the inner products relative to this basis are
$\hbox{diag}(0,1,1,1)$, $\hbox{diag}(0,0,-1,1)$, $\hbox{diag}(0,2,0,1)$ and
$\hbox{diag}(0,1,1,0)$ respectively. In this example we have had the good luck that after modding out $\rad(\F)$, the restrictions of the inner products to $W$ have been in the conditions of the nondegenerate case. However, as we will see, it is not necessary to be lucky in order to solve the problem successfully. 
\end{example}

Let $V$ be a $\K$-vector space. For a subset $X\subseteq V$ we denote by $\span(X)$ the $\K$-linear span of $X$.

\begin{definition}{\rm
Let $\F$ and $\F'$ be two families of  inner products over the same $\K$-vector space $V$, that is, $\F, \, \, \F' \subset \L^2(V \times V;\K)$. We say that $\F$ and $\F'$ are \emph{equivalent} and we write $\F\sim \F' $ if and only if  $\span(\F)=\span(\F')$.
}
\end{definition}

Observe that if $\F\sim\F'$, then a basis $B$ of $V$ orthogonalizes $\F$ if and only if $B$ orthogonalizes $\F'$.

Next we observe that under mild hypothesis on the nature of the ground field, the fact that $\rad(\F)=0$ implies the existence of a nondegenerate inner product in a family  $\F'$  \so\  with  $\F'\sim \F$:

\begin{theorem}\label{notluf}
Assume that $\K$ is an infinite field and $\F$ a family of \so\ inner products in a finite-dimensional $\K$-vector space $V$ such that $\rad(\F)=0$. Then there is  a family $\F'$   with $\F \sim \F'$ such that $\F'$ has a nondegenerate inner product.
\end{theorem}
\begin{proof}
Since $V$ is finite-dimensional, without loss of generality, we may assume that $\F$ is finite because it is equivalent to a finite family $\F'$.  
So to fix ideas write $\F=\{\esc{\cdot,\cdot}_i\}_{i=1}^n$.
If we take any collection of scalars $\lambda_1,\ldots,\lambda_n\in\K$ we can construct the inner product $\escd{\cdot,\cdot}:=\sum_1^n \l_i\esc{\cdot,\cdot}_i$. Then
$\F\cup\{\escd{\cdot,\cdot}\}$ is \so\ if and only if $\F$ is. 
We will replace $\F$ with $\F\cup\{\escd{\cdot,\cdot}\}$ and prove that $\escd{\cdot,\cdot}$ is nondegenerate for some values of $\lambda_1,\ldots,\lambda_n\in\K$. Assume on the contrary that $\escd{\cdot,\cdot}$ is degenerate for any choice of $\lambda_1,\ldots,\lambda_n\in\K$. Then the determinant of the Gram matrix of $\escd{\cdot,\cdot}$ is zero. Since $\F$ is \so\ the matrices
of $\esc{\cdot,\cdot}_i$ are diagonal relative to some basis $B=\{v_1,\ldots,v_m\}$ of $V$. So
the matrix of $\esc{\cdot,\cdot}_i$ is $\hbox{diag}(a_{i1},\ldots,a_{im})$. Consequently the
matrix of $\escd{\cdot,\cdot}$ is $\hbox{diag}(\sum_i \l_i a_{i1},\ldots,\sum_i\l_i a_{im})$. Since the determinant of the Gram matrix of $\esc{\cdot,\cdot}$ is zero we get 
$$\prod_j(\sum_i\l_i a_{ij})=0,$$
for any $(\l_1,\ldots,\l_n)\in\K^n$. Consider the polynomial algebra $\K[x_1,\ldots,x_n]$ in the $n$-indeterminates $x_1,\ldots,x_n$.
The polynomial $p\in\K[x_1,\ldots,x_n]$ given by $p=\prod_j(\sum_i x_i a_{ij})$ vanishes everywhere  (for any values of the variables in $\K$). Since $\K$ is an infinite field, following \cite[Section 8.1.3, item(8), Chapter 8]{Ash} or \cite[Section 1.3, item(7), Chapter 1]{Fulton} we have $p\in I(A^n)=0$ (here $A^n$ is the $n$-dimensional afin space $\K^n$). We get $p=0$ and since $p$ is the product of the homogeneous polynomials $q_j:=\sum_i x_i a_{ij}$ some of these factors must be $0$.
But if some $q_j=0$, then $a_{ij}=0$ for any $i$. Denoting by
$\xi_B(x)$ the coordinates of any $x\in V$ relative to $B$ we have 
$$\esc{v_j,x}_i = \xi_B(v_j)\hbox{diag}(a_{i1},\ldots,a_{im})\xi_B(x)^t=$$   
$$ (\underbrace{0,\ldots,1}_j,0,\ldots, 0)\hbox{diag}(a_{i1},\ldots,a_{im})\xi_B(x)^t=(\underbrace{0,\ldots,a_{ij}}_j,0,\ldots, 0)\xi_B(x)^t=0 
$$
for any $i$ and $x\in V$. Thus $v_j\in\rad(\F)=0$ a contradiction.
So in the family $\F\cup\{\escd{\cdot,\cdot}\}$ there is a nondegenerate inner product. Note that any basis which is orthogonal for all $\F$ is also orthogonal for  $\F\cup\{\escd{\cdot,\cdot}\}$ and conversely.
\end{proof}
As a consequence of the proof given in Theorem \ref{notluf} we have this result.

\begin{corollary}\label{otiel}
If $\K$ is an infinite field, and $\F$ is a \so\ family of inner products in the finite-dimensional vector space $V$ over $\K$ with $\rad(\F)=0$, then $\F$ can be enlarged by adding at most one more inner product linear combination of those in $\F$, so that the new family has a nondegenerate inner product.
\end{corollary}
The hypothesis that $\K$ must be infinite in Corollary \ref{otiel} is essential as the following example shows. 

\begin{example} \rm Consider $\K={\mathbf F}_2$ the field of two elements and the $\mathbf F_2$-vector space $V={\mathbf F_2}^3$. Let $\F$ be the family of inner products whose Gram matrices relative to the canonical basis of $V$ are 
$$\tiny\left(
\begin{array}{ccc}
 0 & 1 & 0 \\
 1 & 1 & 1 \\
 0 & 1 & 0 \\
\end{array}
\right),\quad \left(
\begin{array}{ccc}
 1 & 0 & 1 \\
 0 & 1 & 1 \\
 1 & 1 & 0 \\
\end{array}
\right).$$
It can be checked that $\rad(\F)=0$ but any linear combination of those two inner product has matrix 
$$\tiny\left(
\begin{array}{ccc}
 y & x & y \\
 x & x+y & x+y \\
 y & x+y & 0 \\
\end{array}
\right)$$
whose determinant is $xy(x+y)$ and vanishes for all $x,y \in\mathbf F_2$. So, the family $\F$ is degenerate. Note that $\F$ is \so\ for the basis $B=\{ (1,1,1), (1,0,1), (0,1,1)\}$.
\end{example}

\begin{corollary}
Let $\K$ be an infinite field, and $\F=\{\esc{\cdot,\cdot}_i\}_{i=1}^n$  a family of inner products in the finite-dimensional vector space $V$ over $\K$. Assume that $\rad(\F)=0$ and that for any 
$(\l_1,\ldots,\l_n)\in\K^n$ the inner product 
$\escd{\cdot,\cdot}:=\sum_{i=1}^n\l_i\esc{\cdot,\!\cdot}_i$ is degenerate.
Then $\F$ is not \so.
\end{corollary}

\def\p#1{\esc{\cdot,\!\cdot}_{#1}}

\begin{definition}\rm 
Let $V$ be a vector space over a field $\K$. If  $\F=\{\esc{\cdot,\!\cdot}\}_{i\in I}$ is a family of inner products on  $V$ labelled by $I$, a subspace $S\subset V$ is said to be \emph{$\F$-supplemented} if
there exists a subspace $S'\subset V$ such that $V=S\oPerp_{\F} S'$. The subspace $S'$ is said to be an \emph{$\F$-supplement} of $S$.
\end{definition}

For instance, let $V$ be a $\K$-vector space of arbitrary dimension and let $\F=\{\esc{\cdot,\!\cdot}\}_{i\in I}$ be a family of inner products on $V$ which is simultaneously orthogonalizable. Then for any fixed $i$ the radical $\rad(\p{i})$ is $\F$-supplemented. To see this, consider a basis $\{e_j\}_{j\in J}$ of $V$ which is orthogonal relative to $\F$. Then
$\rad(\p{i})$ is the linear span of all $e_j$'s such that
$\esc{e_j,e_j}_i=0$. So define $S'$ as the linear span of the remaining $e_j$'s. We have $V=\rad(\p{i})\oPerp_{\F} S'$. Thus:


\begin{proposition}\label{supl}
Let $V$ be an arbitrary-dimensional $\K$-vector space and let $\F=\{\esc{\cdot,\!\cdot}_i\}_{i\in I}$ be a family of inner products on $V$.
A necessary condition for $\F$ to be simultaneously orthogonalizable is that the radical of each inner product of $\F$ be $\F$-supplemented. Furthermore, for any basis $\{e_j\}_{j\in J}$ simultaneously orthogonalizing $\F$, and any $i\in I$ there is subset $J_i\subset J$ such that 
$\{e_j\}_{j\in J_i}$ is a basis of $\rad(\p{i})$.
\end{proposition}

The following lemma shows the relationship between $\F$ and $\F \vert_{S}$.

\begin{lemma} \label{elpus} 
Let $V$ be a $\K$-vector space of arbitrary dimension. Let $\F=\{\esc{\cdot,\!\cdot}_i\}_{i\in I}$ be a family of inner products. Let  $S$ be a subspace of $V$ and suppose that there exists $S'$ a $\F$-supplement. Then:
\begin{enumerate}
    \item \label{moka} If $\F \vert_ S$ and $\F \vert_ {S'}$ are \so\ \!\!, then $\F$ is \so.
    \item If $\rad(\F)=0$, then $\rad(\F \vert_{S'})=\rad(\F \vert_{S})=0$.
\end{enumerate}
\end{lemma}

We observe that the converse of item \eqref{moka} of Lemma \ref{elpus}  is not true in general as the following example shows:

\begin{example}{\rm
We consider a field $\K$ of characteristic $2$ and $V=\K^3$ with $\F=\{\esc{\cdot,\cdot}\}$ being $\esc{\cdot,\cdot}$ the inner product whose matrix in the canonical basis $\{e_i\}_{i=1}^3$ is:
$$\small\begin{pmatrix}1 & 0 & 0\cr 0 & 0 & 1\cr 0 & 1 & 0\end{pmatrix}.$$
Now decompose $V=\K e_1\oPerp(\K e_2\oplus \K e_3)$. The subspace $\K e_2\oplus \K e_3$ is not orthogonalizable (since any vector is isotropic), however $V$ has an orthogonal basis $\{e_1+e_2+e_3,e_1+e_2,e_1+e_3\}$.
}
\end{example}

In view of Proposition \ref{supl} and Lemma \ref{elpus}  we may expect to construct a basis orthogonalizing $\F$ from basis of $\rad(\p{i})$ which orthogonalizes the remaining radicals $\rad(\p{j})$ (with $j\ne i$). In particular one of the cases where the simultaneous orthogonalization is inherited by orthogonal summands is given.

\begin{theorem}\label{erdamus}
Let $\F=\{\esc{\cdot,\!\cdot}\}_{i\in I}$ be a family of inner products in an arbitrary-dimensional vector space $V$ over a field $\K$ and assume that $\rad(\F)=0$. 
Then the following assertions are equivalent:
\begin{enumerate}
    \item  $\F$ is simultaneously orthogonalizable.
    \item There is an $i\in I$ such that 
$\rad(\p{i}) \ne V$ and  $\rad(\p{i})$ is $\F$-supp\-lemented  for some supplement $S'$ such that: (i) the family of inner products $\F\vert_{S'}$ is simultaneously orthogonalizable and nondegenerate; (ii)  the family of inner products  $\F\vert_{\rads(\p{i})}$ is simultaneously orthogonalizable.

\end{enumerate}
\end{theorem}

\begin{proof}
 Let $\{e_j\}_{j\in J}$ be an orthogonal basis of $V$ relative to the family $\F$. Since $\rad(\F)=0$ there is some $i\in I$ such that $\rad(\p{i})\ne V$. Applying Proposition \ref{supl} we know that $\rad(\p{i})$ is $\F$-supplemented. 
Then $\rad(\p{i})$ is the linear span of a certain subset of the previous basis:
 $$\rad(\p{i})=\sp\{e_j\}_{j\in J_i}$$ for some $J_i\subset J$. Then $S'=\sp\{e_j\}_{j\in J\setminus J_i}$
 is an $\F$-supplement of $\rad(\p{i})$ and the family of inner products $\F\vert_{S'}$ is simultaneously orthogonalizable (relative to the basis 
 $\{e_j\}_{j\in J\setminus J_i}$) and nondegenerate. Observe that $\F\vert_{\rads(\p{i})}$ is simultaneously orthogonalizable too. For the converse apply Lemma \ref{elpus}.
\end{proof}

\begin{remark} \rm
Without loss of generality, in the family $\F\vert_{\rads(\p{i})}$ we can eliminate $\p{i} \vert_{\rads(\p{i})}$  since this inner product is null. At a computational level this simplifies the complexity of the problem.
\end{remark}

\begin{theorem}\label{tsal}
Let $V$ be a finite-dimensional vector space  over  a field $\K$ and let $\F=\{\p{i}\}_{i\in I}$ be a family of inner products. Then  $\F$ is \so\ if and only if $V =\rad(\F)\oPerp \left( \oPerp_{j \in J} V_j \right)$ with $|J| < \infty$ such that $\F\vert_{V_j} $ is nondegenerate and \so. 
\end{theorem}

\begin{proof}
The non trivial implication is as follows. First, we apply Proposition \ref{tomato} which gives a decomposition $V =\rad(\F)\oplus W$ where $\F\vert_{W}$ is \so\ and has zero radical. Next, we apply Theorem \ref{erdamus} to the subspace $W$ and the family of inner products $\F\vert_{W}$ and we repeat this process. Since the dimension of $V$ is finite this process finishes in a finite number of steps. 
\end{proof}

\begin{example}\label{otxes}{\rm
As a final example we consider $\K=\Q$ and $V={\Q}^6$, being  the family of inner products $\F=\{\p{i}\}_{i=1}^6$ whose Gram matrices relative to the canonical basis are $\{B_i\}_{i=1}^6$ (respectively) given as follows:
\vskip0.5cm

\begin{center}
\scalebox{0.8}{
\begin{tabular}{l l}
$B_1=
\left(
\begin{array}{rrrrrr}
 1 & 0 & -1 & 1 & 0 & 1 \\
 0 & 2 & -3 & -1 & -1 & -1 \\
 -1 & -3 & 2 & -2 & 2 & 1 \\
 1 & -1 & -2 & 0 & 0 & 1 \\
 0 & -1 & 2 & 0 & 3 & 3 \\
 1 & -1 & 1 & 1 & 3 & 4 \\
 \end{array} \right)$, &
$B_2=\left(
\begin{array}{rrrrrr}
 2 & 1 & 1 & 0 & 2 & 2 \\
 1 & 3 & -1 & 0 & 1 & 1 \\
 1 & -1 & 2 & 0 & 2 & 2 \\
 0 & 0 & 0 & 0 & 0 & 0 \\
 2 & 1 & 2 & 0 & 4 & 4 \\
 2 & 1 & 2 & 0 & 4 & 4 \\
 \end{array}\right)$,
\\
\\
 $B_3=\left(
\begin{array}{rrrrrr}
 3 & 1 & 0 & 1 & 2 & 3 \\
 1 & 1 & -1 & -1 & 2 & 2 \\
 0 & -1 & 1 & -2 & 1 & 0 \\
 1 & -1 & -2 & 0 & 0 & 1 \\
 2 & 2 & 1 & 0 & 3 & 3 \\
 3 & 2 & 0 & 1 & 3 & 4 \\
\end{array}\right)$, &
$B_4=\left(
\begin{array}{rrrrrr}
 1 & 1 & 2 & -1 & 2 & 1 \\
 1 & 1 & -2 & -1 & 0 & 0 \\
 2 & -2 & 0 & 0 & 2 & 3 \\
 -1 & -1 & 0 & -2 & 0 & -1 \\
 2 & 0 & 2 & 0 & 3 & 3 \\
 1 & 0 & 3 & -1 & 3 & 2 \\
\end{array}\right)$,\\ \\

$B_5=\left(
\begin{array}{rrrrrr}
 0 & -1 & 1 & 0 & 0 & 0 \\
 -1 & 2 & -2 & 0 & -2 & -2 \\
 1 & -2 & 3 & 0 & 3 & 3 \\
 0 & 0 & 0 & 0 & 0 & 0 \\
 0 & -2 & 3 & 0 & 3 & 3 \\
 0 & -2 & 3 & 0 & 3 & 3 \\
\end{array}\right)$, &

$B_6=\left(
\begin{array}{rrrrrr}
 2 & 0 & 0 & 1 & 1 & 2 \\
 0 & -1 & 0 & -1 & 2 & 2 \\
 0 & 0 & 0 & -2 & 0 & -1 \\
 1 & -1 & -2 & 0 & 0 & 1 \\
 1 & 2 & 0 & 0 & 1 & 1 \\
 2 & 2 & -1 & 1 & 1 & 2 \\
\end{array}\right)$.
\end{tabular}}
\end{center}

\vskip0.5cm

\noindent
We can see that $\rad(\F)=0$ and $\rad(\p{1})=\K e_1\oplus\K e_2$ where the vectors are
$e_1=(-5,-1,-2,3,1,0)$ and $e_2=(-6,-1,-2,3,0,1)$. It can be checked that $\F\vert_{\rads(\p{1})}$
is nondegenerate, so applying Theorem \ref{nomam} it has a \so\ basis.
Thus we find a basis of $\rad(\p{1})$ orthogonal relative to $\F\vert_{\rads(\p{1})}$: concretely $f_1= -6 e_1 + 5 e_2$,
$f_2= 7 e_1 - 6 e_2$. Next we check that $\F\vert_{\rads(\p{1})}$ is supplemented and we give such a supplement $S'=\oplus_{i=3}^6 \K f_i$ where $f_3=(-1,1,1,0,0,0)$, $f_4=(0,0,0,1,0,0)$,  $f_5=(-1,0,0,0,1,0)$ and $f_6=(-1,0,0,0,0,1)$. Now, again by
Theorem \ref{nomam}, $\F\vert_{S'}$ is nondegenerate and we find a \so\ basis of $S'$  given  by $\{q_3,q_4,q_5,q_6\}$ where
$q_3=(0,0,0,0,1,-1)$, $q_4=(0,0,0,-1,-1,1)$, $q_5=(-1,1,1,-2,-4,4)$ and $q_6=(0,1,1,-2,-4,3)$.
So, definitively there is a basis of $V$ given by $\{f_1,f_2,q_3,q_4,q_5,q_6\}$ and
the inner products in this basis are given by the matrices:
$\diag(0,0,1,-1,1,2)$, $\diag(1,1,0,0,1,1)$, $\diag(1,1,1,-1,1,0)$, $\diag(1,1,-1,-1,0,1)$, $\diag(1,-1,0,0,1,2)$ and $\diag(1,0,1,-1,1,-1)$. 

Let us mention that if we have an algebra structure on $\Q^6$ with structure constants
$c_{ijk}$ being the $(i,j)$ entry of $B_k$, then this algebra is an evolution algebra
and a natural basis is precisely $\{f_1,f_2,q_3,q_4,q_5,q_6\}$ begin its structure matrix

$$\tiny\left(
\begin{array}{rrrrrr}
0  & 1 & 1   &  1  & 1 & 1\\
0  & 1 & 1   &  1  & -1 & 0\\ 
1  & 0 & 1   &  -1 & 0 & 1 \\
-1 & 0 & -1  &  -1 & 0 & -1 \\
1  & 1 &  1  &  0   & 1 & 1 \\
2  & 1 &  0  &  1  &  2 & -1\\
\end{array}\right).
$$

}
\end{example}

\end{document}